\documentclass[10pt,reqno]{amsart}
\usepackage{amsmath}
\usepackage{amssymb}
\usepackage{amsthm}

\newlength{\defbaselineskip}
\newcommand{\setlinespacing}[1]%
           {\setlength{\baselineskip}{#1 \defbaselineskip}}

\numberwithin{equation}{section}

\newtheorem{thm}{Theorem}[section]

\newtheorem{lem}[thm]{Lemma}
\newtheorem{prop}[thm]{Proposition}

\theoremstyle{definition}

\theoremstyle{remark}
\newtheorem{rem}[thm]{Remark}
\numberwithin{equation}{section}

\begin{document}
\title[Local-in-time Strichartz estimates]{On local-in-time Strichartz estimates for the Schr\"odinger  equation with singular potentials}

\author{Seongyeon Kim, Ihyeok Seo and Jihyeon Seok}

\thanks{This research was supported by NRF-2019R1F1A1061316.}


\address{Department of Mathematics, Sungkyunkwan University, Suwon 16419, Republic of Korea}

\email{synkim@skku.edu}

\email{ihseo@skku.edu}

\email{jseok@skku.edu}

\begin{abstract}
There have been a lot of works concerning the Strichartz estimates for the perturbed Schr\"odinger equation by potential.
These can be basically carried out adopting the well-known procedure for obtaining the Strichartz estimates
from the weighted $L^2$ resolvent estimates for the Laplacian.
In this paper we handle the Strichartz estimates without relying on the resolvent estimates.
This enables us to consider various potential classes such as the Morrey-Campanato classes.
\end{abstract}

\maketitle

\section{Introduction}
In this paper we are concerned with Strichartz estimates for the Schr\"odinger equation with perturbations by potential,
\begin{equation}\label{SE}
\begin{cases}
i\partial_{t} u+\Delta u+V(x)u = 0,\\
u(x,0)=u_0(x),
\end{cases}
\end{equation}
where $u:\mathbb{R}^n\times\mathbb{R}\rightarrow\mathbb{C}$ and $V:\mathbb{R}^n\rightarrow\mathbb{C}$.
The Strichartz estimates for  \eqref{SE}
have been extensively studied over the past several decades.

In the free case $V=0$, it was first established by Strichartz \cite{Str} in connection with the Tomas-Stein restriction theorem (\cite{T,St})
in harmonic analysis:
\begin{equation*}
\|u\|_{L^{\frac{2(n+2)}{n}}(\mathbb{R}^{n+1})} \lesssim \|u_0\|_{L^2}.
\end{equation*}
Since then, there have been developments in extending this estimate to mixed norms $L_t^qL_x^r$
(see \cite{GV,KT,M-S}):
\begin{equation}\label{two-s}
\|u\|_{L_t^q(\mathbb{R}; L_x^r(\mathbb{R}^n))} \lesssim \|u_0\|_{L^2}
\end{equation}
if and only if $(q,r)$ is Schr\"odinger-admissible; $q\ge 2$, $2/q +n/r =n/2$ and $(q,r,n) \neq (2, \infty, 2)$.

For the perturbed \eqref{SE}, the decay $|V|\sim|x|^{-2}$ has been known to be critical
for which the Strichartz estimate \eqref{two-s} holds (see \cite{GVV,D}),
and recent studies have intensively aimed to get as close as possible to the inverse-square potential.
Rodnianski and Schlag \cite{RS} obtained \eqref{two-s} for non-endpoint Schr\"odinger-admissible pairs ($q>2$)
with almost critical decay $|V|\lesssim(1+|x|)^{-2-\varepsilon}$.
In \cite{BPST-Z,BPST-Z2}, this decay assumption is weaken to the critical $|x|^{-2}$ including the endpoint case $q=2$.
For the non-endpoint case $q> 2$, these results are recently generalized in \cite{BM} to weak $L^{n/2}$ potentials which include $|x|^{-2}$,
and generalized earlier in \cite{KK} to more larger Morrey-Campanato classes $\mathcal{L}^{2,p}(\mathbb{R}^n)$ for $p>(n-1)/2$, $n\geq3$.
In general, the Morrey-Campanato class $\mathcal{L}^{\alpha,p}(\mathbb{R}^n)$ is defined for $\alpha>0$ and $1\leq p\leq n/\alpha$ by
$$\|V\|_{\mathcal{L}^{\alpha,p}}:=
\sup_{x\in\mathbb{R}^n, r>0} r^{\alpha} \bigg( r^{-n}\int_{B(x,r)} |V(y)|^p dy\bigg)^{1/p} < \infty,$$
where $B(x,r)$ denotes a ball in $\mathbb{R}^n$ centered at $x$ with radius $r$.

Such generalizations with small potentials can be carried out adopting the well-known procedure (\textit{cf}. \cite{BPST-Z2,BM})
for obtaining the Strichartz estimates
from the weighted $L^2$ resolvent estimates of the form
\begin{equation}\label{resol}
\|(-\Delta-z)^{-1}f\|_{L^2(w(x))}\lesssim C(w)\|f\|_{L^2(w(x)^{-1})}
\end{equation}
where $z\in\mathbb{C}\setminus[0,\infty)$
and a weight $w:\mathbb{R}^{n}\rightarrow[0,\infty]$ is a locally integrable function
which is allowed to be zero or infinite only on a set of Lebesgue measure zero,
so $w^{-1}$ is also a weight if it is locally integrable.
In particular, we can obtain
\begin{equation}\label{3d}
\|u\|_{L_t^q(\mathbb{R}; L_x^r(\mathbb{R}^3))} \lesssim \|u_0\|_{L^2}
\end{equation}
for small potentials $V\in\mathcal{KS}_2(\mathbb{R}^3)$ by making use of \eqref{resol} obtained in \cite{BBRV} with $C(w)\sim\|w\|_{\mathcal{KS}_2(\mathbb{R}^3)}$.
In general, the Kerman-Sawyer class $\mathcal{KS}_{\alpha}(\mathbb{R}^n)$ is defined for $0<\alpha<n$ by
$$ \|V\|_{\mathcal{KS}_{\alpha}}:=\sup_Q \left( \int_Q |V(x)| dx \right)^{-1} \int_Q \int_Q \frac{|V(x)||V(y)|}{|x-y|^{n-\alpha}}  dxdy <\infty,$$
where the sup is taken over all dyadic cubes\footnote{Here the sup is taken over all cubes in $\mathbb{R}^n$ of sidelength $2^{-k}$
	and corners in the set $2^{-k}\mathbb{Z}^n$ for all integers $k$.} $Q$ in $\mathbb{R}^n$.
This class is closely related to the global Kato ($\mathcal{K}$) and Rollnik ($\mathcal{R}$) classes
which are fundamental in spectral and scattering theory (\textit{cf}. \cite{K, Si}) and was revealed in \cite{RS} for their usefulness for dispersive properties of the Schr\"odinger equation. They are defined by
\begin{equation*}
V \in \mathcal{K}\quad \Leftrightarrow \quad \|V\|_{\mathcal{K}} := \sup_{x \in \mathbb{R}^n} \int_{\mathbb{R}^n} \frac{|V(y)|}{|x-y|^{n-2}} dy <\infty
\end{equation*}
and
\begin{equation*}
V \in \mathcal{R} \quad \Leftrightarrow \quad \|V\|_\mathcal{R} = \left( \int_{\mathbb{R}^3}\int_{\mathbb{R}^3} \frac{|V(x)V(y)|}{|x-y|^2}dxdy \right)^{1/2} < \infty.
\end{equation*}
In particular, we note that $\mathcal{K},\mathcal{R} \subset \mathcal{KS}_2$ and $\mathcal{L}^{\alpha,p}\subset\mathcal{KS}_{\alpha}$ for $p>1$.

Motivated by these inclusions, we desire to obtain \eqref{3d} for general dimensions.
However, the method in \cite{BBRV} is available only in three dimensions and \eqref{resol} is left unsolved
when $w\in\mathcal{KS}_2(\mathbb{R}^n)$ for $n\neq3$.
In this regard, we aim to obtain a higher dimensional version of \eqref{3d} without relying on the resolvent estimates.
We instead make use of weighted $L_t^qL_x^2$ estimates with $q>2$ for the Schr\"odinger flow $e^{it\Delta}$.
In addition to the Strichartz estimates, we also show that the Cauchy problem \eqref{SE} can be locally well-posed
under the potentials considered here.
Our result is the following theorem.

\begin{thm}\label{SS}
	Let $n \ge 3$, $ n-1< \alpha < n$ and $2<q<a<\infty$.
	Assume that $\|V^{\frac{\alpha q(a-2)}{4(a-q)}}\|_{\mathcal{KS}_\alpha}^{\frac{4(a-q)}{\alpha q(a-2)}}$
is finite.\footnote{Hereafter, we shall take fractional powers of $V$ without putting any absolute value
since $|V^\lambda|=|V|^\lambda$ for $V\in\mathbb{C}$ and $\lambda\in\mathbb{R}$.}
	For any $T>0$ there exists then a unique solution $u \in  C_t([0,T]; L_x^2(\mathbb{R}^n)) \,\cap\, L_t^q([0,T]; L_x^2(|V|))$ to \eqref{SE}
	with $u_0 \in L^2$.
Furthermore,
	\begin{equation}\label{ss11}
	\sup_{t\in[0,T]}\|u(x,t)\|_{L_x^2}+\|u(x,t)\|_{L_t^q([0,T]; L_x^2(|V|))}\lesssim\|u_0\|_{L^2},
	\end{equation}
and if\, $n/r = n/2-2/q$
	\begin{equation}\label{ss}
	\|u\|_{L_t^q([0,T]; L_x^r(\mathbb{R}^n))} \lesssim \|u_0\|_{L^2}.
	\end{equation}
\end{thm}

\begin{rem}
	The quantity $\|V^{\beta}\|_{\mathcal{KS}_{\alpha}}^{1/\beta}$ already appeared in \cite{S,S2,LS,C} concerning various problems for the Schr\"odinger equation as well as the Schr\"odinger operator.
	From the inclusion, if $p\neq1$, we also see that
	$$\|V^{\frac{\alpha q(a-2)}{4(a-q)}}\|_{\mathcal{KS}_\alpha}^{\frac{4(a-q)}{\alpha q(a-2)}} \le \|V^{\frac{\alpha q(a-2)}{4(a-q)}}\|_{\mathcal{L}^{\alpha, p}}^{\frac{4(a-q)}{\alpha q(a-2)}}=\|V\|_{\mathcal{L}^{\frac{4(a-q)}{q(a-2)}, \frac{p\alpha q(a-2)}{4(a-q)}}}.$$
	Hence the theorem includes various new Morrey-Campanato classes $\mathcal{L}^{\gamma,s}$ with $\gamma<2$ and $s>\alpha/\gamma$.
For example, if $V(x)=\phi(\frac{x}{|x|})|x|^{-\gamma}$, $\phi\in L^s(\mathbb{S}^{n-1})$, $\frac{n-1}{\gamma}<s<\frac{n}{\gamma}$,
then $V$ need not belong to $L^{n/\gamma,\infty}$, but $V\in\mathcal{L}^{\gamma,s}$.
	More interestingly, this shows that
	the decay assumption on the potential can be weaken more by $|x|^{-\gamma}$ with $\gamma<2$ than the inverse square potential
	when considering the Strichartz estimates locally in time.
\end{rem}

The theorem also holds for time-dependent potentials $V(x,t)$ with suitable assumptions.
For example, the above assumptions on $V(x)$ replaced by $W(x)=\sup_{t\in[0,T]}|V(x,t)|$
would suffice. Taking the supremum in this way already appeared in the study of local regularity of the solution (\cite{RV,RV2}).
For more general time-dependent potentials,
weighted $L_{x,t}^2$ estimates have been developed by Koh and the second author \cite{KS,KS2},
and therefore the potentials therein would also suffice.

\

Throughout this paper, the letter $C$ stands for a positive constant which may be different
at each occurrence.
We also denote $A\lesssim B$ to mean $A\leq CB$
with unspecified constants $C>0$.

\section{Weighted Strichartz estimates}\label{sec2}

In this section we obtain weighted $L_t^qL_x^2$ Strichartz estimates with $q>2$ for the free Schr\"odinger flow,
which will be used for the proof of Theorem \ref{SS} in the next section:

\begin{prop}\label{sp}
	Let $n \ge 3$ and $ n-1< \alpha < n$. If $2<q<a<\infty$ then we have
	\begin{equation}\label{shomo}
	\|e^{it\Delta}f\|_{L_t^q([0,T]; L_x^2(w))}\lesssim T^{\frac{q-2}{q(a-2)}}
	\|w^{\frac{\alpha q(a-2)}{4(a-q)}}\|_{\mathcal{KS}_{\alpha}}^{\frac{2(a-q)}{\alpha q(a-2)}}\|f\|_{L^2}
	\end{equation}
	and
	\begin{align}\label{sinhomo}
	\nonumber\bigg\|\int_{-\infty}^{t} e^{i(t-s)\Delta}F(\cdot, s) &ds\bigg\|_{L_t^q([0,T]; L_x^2(w))} \\
	\lesssim &\,\, T^{\frac{2(q-2)}{q(a-2)}}\|w^{\frac{\alpha q(a-2)}{4(a-q)}}\|_{\mathcal{KS}_{\alpha}}^{\frac{4(a-q)}{\alpha q(a-2)}} \|F\|_{L_t^{q\prime}([0,T];L_x^2(w^{-1}))}.
	\end{align}
\end{prop}

\begin{proof}
	To show the first estimate \eqref{shomo}, we first note that
	\begin{equation}\label{energy}
	\|e^{it\Delta}f \|_{L_t^{a}([0,T]; L_x^2)} \le T^{1/a}\|f\|_{L^2}
	\end{equation}
	with $2<a<\infty $ and
	\begin{equation}\label{swl2}
	\big\|e^{it\Delta}f\big\|_{L_t^2([0,T];L_x^2(w))} \lesssim \|w^{\frac{\alpha}{2}}\|_{\mathcal{KS}_{\alpha}}^{\frac{1}{\alpha}} \| f\|_{L^2}
	\end{equation}
	for $n-1<\alpha<n$ with $n\geq3$.
	Estimate \eqref{energy} follows immediately from H\"older's inequality and the fact that $e^{it\Delta}$ is isometry on $L^2$;
	\begin{equation*}
	\|e^{it\Delta}f \|_{L_t^{a}([0,T]; L_x^2)} \le T^{1/a}\|e^{it\Delta}f\|_{L_t^{\infty}L_x^2} \le T^{1/a}\|f\|_{L^2}.
	\end{equation*}
	On the other hand, \eqref{swl2}
	was already obtained by the first author in \cite{S2}.
	We now deduce \eqref{shomo} from applying the complex interpolation with Lemma \ref{rint} below,
	between the two estimates \eqref{energy} and \eqref{swl2}.
	
	\begin{lem}[\cite{BL}]\label{rint}
		Let $0 < \theta<1$ and $1 \le q_0, q_1 < \infty$.
		Given two complex Banach spaces $A_0$ and $A_1$,
		$$(L^{q_0}(A_0), L^{q_1}(A_1))_{[\theta]} = L^q((A_0,A_1)_{[\theta]})$$
		if $1/q=(1-\theta)/q_0 +\theta/q_1$, and if $w=w_0^{1-\theta}w_1^{\theta}$
		$$(L^{2}(w_0), L^2(w_1))_{[\theta]} =L^2(w).$$	
		Here, $(\cdot\,,\cdot)_{[\theta]}$ denotes the complex interpolation functor.
	\end{lem}

	Indeed, using the complex interpolation between \eqref{energy} and \eqref{swl2}, we first see
	\begin{equation*}
	\|e^{it\Delta}f\|_{(L_t^a([0,T];L_x^2),L_t^2([0,T];L_x^2(w)))_{[\theta]}}  \lesssim T^{\frac{1-\theta}a}\|w^{\frac{\alpha}{2}}\|_{\mathcal{KS}_{\alpha}}^{\frac{\theta}{\alpha}}\|f\|_{L^2},
	\end{equation*}
	and then we make use of Lemma \ref{rint} to get
	\begin{equation*}
	\|e^{it\Delta}f\|_{L_t^q([0,T]; L_x^2(w^{\theta}))}\lesssim T^{\frac{1-\theta}a}\|w^{\frac{\alpha}{2}}\|_{\mathcal{KS}_{\alpha}}^{\frac{\theta}{\alpha}}\|f\|_{L^2}
	\end{equation*}
	where
	\begin{equation*}
	\frac{1}{q} = \frac{1-\theta}{a}+\frac{\theta}{2},\quad 2<q<a <\infty, \quad 0<\theta<1,\quad  n-1 <\alpha <n, \quad n \ge 3.
	\end{equation*}
	By substituting $\theta=\frac{2(a-q)}{q(a-2)}$, we have
	\begin{equation*}
	\|e^{it\Delta}f\|_{L_t^q([0,T]; L_x^2(w^{\frac{2(a-q)}{q(a-2)}}))}\lesssim T^{\frac{q-2}{q(a-2)}}\|w^{\frac{\alpha}{2}}\|_{\mathcal{KS}_{\alpha}}^{\frac{2(a-q)}{\alpha q(a-2)}}\|f\|_{L^2}
	\end{equation*}
	under $2<q<a <\infty$, $n-1 <\alpha <n$ and $n \ge 3$. Replacing $w^{\frac{2(a-q)}{q(a-2)}}$ with $w$,
	we obtain the desired estimate \eqref{shomo}.

	Next, we show the estimate \eqref{sinhomo}. By the standard $TT^*$ argument, \eqref{shomo} implies
	$$
	\bigg\|\int_{-\infty}^{\infty} e^{i(t-s)\Delta}F(\cdot, s) ds\bigg\|_{L_t^q([0,T]; L_x^2(w))}
	\lesssim  T^{\frac{2(q-2)}{q(a-2)}}\|w^{\frac{\alpha q(a-2)}{4(a-q)}}\|_{\mathcal{KS}_{\alpha}}^{\frac{4(a-q)}{\alpha q(a-2)}} \|F\|_{L_t^{q\prime}([0,T];L_x^2(w^{-1}))}.
	$$
	Since $q>q'$ from $q>2$, we use the Christ-Kiselev lemma (\cite{CK}) to conclude \eqref{sinhomo}.
\end{proof}

\section{Proof of Theorem \ref{SS}}
In this section we prove Theorem \ref{SS} by making use of the weighted $L_t^qL_x^2$ estimates in the previous section.

We first consider the potential term in \eqref{SE} as a source term and then write the solution of \eqref{SE} as the sum of the solution to
the free Schr\"odinger equation plus a Duhamel term, as follows:
\begin{equation}\label{wer}
\mathcal{T}(u) = e^{it\Delta}u_0 - i \int_0^t e^{i(t-s)\Delta}(V(\cdot)u(\cdot,s))ds.
\end{equation}
Then we apply the estimates obtained in the previous section to each of the terms in \eqref{wer}
to show that $\mathcal{T}$ defines a contraction map on
\begin{align*}
X(T) = \{ u \in C_t([0,T];L_x^2) \cap &L_t^q([0,T];L_x^2 (|V|)) : \\
&\sup_{t \in [0,T]} \|u\|_{L_x^2} + \|u\|_{L_t^q ([0,T];L^2_x(|V|))}< \infty \}
\end{align*}
equipped with the distance
\begin{equation*}
d(u,v) = \sup_{t \in [0,T]}\|u-v\|_{L_x^2} + \|u-v\|_{L_t^q([0,T];L_x^2(|V|))}
\end{equation*}
for an appropriate value of $T>0$.
We shall use the notation $ |||V||| :=\|V^{\frac{\alpha q(a-2)}{4(a-q)}}\|_{\mathcal{KS}_{\alpha}}^{\frac{4(a-q)}{\alpha q(a-2)}}$
for brevity.

Let us now show that $\mathcal{T}$ is well-defined on $X$.
By applying Proposition \ref{sp} with $w=|V|$, we see that
\begin{align} \label{T1}
\nonumber\|\mathcal{T}(u)\|_{L_t^q ([0,T];L^2_x (|V|))}
&\lesssim
||| V |||^{1/2} T^{\frac{q-2}{q(a-2)}} \|u_0\|_{L^2} + ||| V ||| T^{\frac{2(q-2)}{q(a-2)}} \|u\|_{L_t^{q\prime}([0,T];L_x^2(|V|))}\\
&\lesssim
|||V|||^{1/2} T^{\frac{q-2}{q(a-2)}}\big(1+|||V|||T^{\frac{a(q-2)}{q(a-2)}}\big) \|u_0\|_{L^2}
\end{align}
where, for the second inequality, we used
\begin{equation}\label{sd1}
\|u\|_{L_t^{q\prime}([0,T]; L_x^2(|V|))} \lesssim  |||V|||^{1/2}T^{\frac{(a-1)(q-2)}{q(a-2)}} \|u_0\|_{L^2}.
\end{equation}
This follows from using H\"older's inequality in time and then applying Proposition \ref{sp} to \eqref{wer} with $w=|V|$.
Indeed,
\begin{align}\label{abso}
\nonumber\|u\|_{L_t^{q\prime}([0,T]; L_x^2(|V|))}
&\le T^{1-2/q} \|u\|_{L_t^{q}([0,T];L_x^2(|V|))}\\
\nonumber&\lesssim T^{1-2/q}\Big(|||V|||^{1/2}T^{\frac{q-2}{q(a-2)}}\|u_0\|_{L^2} \\&+|||V|||T^{\frac{2(q-2)}{q(a-2)}}\|u\|_{L_t^{q\prime}([0,T]; L_x^2(|V|))}\Big)
\end{align}
for $2<q<a<\infty$, $n-1<\alpha<n$ and $n \ge3$.
Since we are assuming that $|||V|||$ is finite and
the power of $T$ in \eqref{abso} is positive, if $T$ is assumed to be small,
the last term on the right-hand side of \eqref{abso} can be absorbed into the left-hand side.
Hence, we obtain \eqref{sd1} for small $T$.
On the other hand,
\begin{align*}
\sup_{t \in [0,T]} \|\mathcal{T}(u)\|_{L_x^2}
&\lesssim \|u_0\|_{L^2} + \sup_{t\in[0,T]} \bigg\|\int_0^t e^{-is\Delta}(V(\cdot)u(\cdot,s))ds \bigg\|_{L_x^2}\\
&\lesssim \|u_0\|_{L^2} + |||V|||^{1/2} T^{\frac{q-2}{q(a-2)}} \|u\|_{L_t^{q'}([0,T];L_x^2(|V|))}
\end{align*}
using the fact that $e^{it\Delta}$ is an isometry in $L^2$ and then applying the following dual estimate of \eqref{shomo} with $F=\chi_{[0,t]}Vu$
\begin{equation} \label{dual}
\bigg\| \int_{-\infty}^{\infty} e^{-is\Delta} F(\cdot,s)ds\bigg\|_{L_x^2} \lesssim |||V|||^{1/2} T^{\frac{q-2}{q(a-2)}} \|F\|_{L_t^{q'}([0,T];L_x^2(|V|^{-1}))}.
\end{equation}
By \eqref{sd1}, we finally get
\begin{equation} \label{T2}
\sup_{t \in [0,T]} \|\mathcal{T}(u)\|_{L_x^2} \lesssim \big(1+|||V||| T^{\frac{a(q-2)}{q(a-2)}}\big) \|u_0\|_{L^2}.
\end{equation}
Consequently, \eqref{T1} and \eqref{T2} imply $\mathcal{T}(u) \in X$ for $u \in X$.

Next, we show that $\mathcal{T}$ is a contraction on $X$.
Using \eqref{sinhomo} and H\"older's inequality in time, we see that
\begin{align} \label{cont1}
\nonumber
\|\mathcal{T}(u) - \mathcal{T}(v) \|_{L^q_t([0,T];L_x^2(|V|))}
\nonumber
&\lesssim |||V||| T^{\frac{2(q-2)}{q(a-2)}} \|u-v\|_{L^{q'}_t([0,T];L_x^2(|V|))} \\
&\lesssim |||V||| T^{\frac{2(q-2)}{q(a-2)}} T^{\frac{q-2}{q}}\|u-v\|_{L^{q}_t([0,T];L_x^2(|V|))},
\end{align}
while
\begin{align} \label{cont2}
\nonumber
\sup_{t \in [0,T]} \|\mathcal{T}(u)-\mathcal{T}(v)\|_{L^2} &\lesssim \sup_{t \in [0,T]} \bigg\|\int_0^t e^{-is\Delta}(V(\cdot)(u-v)(\cdot,s))ds \bigg\|_{L_x^2} \\
\nonumber
&\lesssim |||V|||^{1/2} T^{\frac{q-2}{q(a-2)}} \|u-v\|_{L_t^{q'}([0,T];L_x^2(|V|))} \\
&\lesssim |||V|||^{1/2} T^{\frac{q-2}{q(a-2)}}T^{\frac{q-2}{q}} \|u-v\|_{L^{q}_t([0,T];L_x^2(|V|))}
\end{align}
using the isometry $e^{it\Delta}$, \eqref{dual} and H\"older's inequalty in time.
Combining \eqref{cont1} and \eqref{cont2}, we obtain for $u,v \in X$
\begin{equation*}
d(\mathcal{T}(u),\mathcal{T}(v)) \leq C|||V|||^{1/2}T^{\frac{(q-2)(a-1)}{q(a-2)}}\big(1+|||V|||^{1/2}T^{\frac{q-2}{q(a-2)}}\big)d(u,v).
\end{equation*}
Hence $\mathcal{T}$ is a contraction on $X$ by taking $T$ sufficiently small.

Now by the contraction mapping principle, there exists a unique solution
$$u \in  C_t([0,T]; L_x^2(\mathbb{R}^n)) \,\cap\, L_t^q([0,T]; L_x^2(|V|))$$
for a small $T>0$.
Moreover, the estimate \eqref{ss11} follows immediately from \eqref{T1} and \eqref{T2}.
Thanks to
\begin{equation}\label{mass}
\sup_{t\in[0,T]}\|u(x,t)\|_{L_x^2}\lesssim\|u_0\|_{L^2},
\end{equation}
the above process can be also iterated on translated time intervals
to extend the above solution to any finite time $T$.
The estimate \eqref{mass} follows from applying the isometry of $e^{it\Delta}$, the dual estimate of \eqref{shomo}, and then \eqref{sd1}.

It remains to prove the Strichartz estimates \eqref{ss}.
We first write the solution to \eqref{SE} as before:
\begin{equation*}
u(x,t)=e^{it\Delta}u_0 -i \int_0^t e^{i(t-s)\Delta}(V(\cdot)u(\cdot, s)) ds.
\end{equation*}
By the classical Strichartz estimates \eqref{two-s} for the free Schr\"odinger flow, we then have
\begin{equation*}
\|u\|_{L_t^q([0,T]; L_x^r)} \lesssim \|u_0\|_{L^2} +\bigg\|\int_0^t e^{i(t-s)\Delta}(V(\cdot)u(\cdot, s)) ds\bigg\|_{L_t^q([0,T]; L_x^r)}
\end{equation*}
for $q\ge 2$ and $2/q +n/r = n/2$.
Now we are reduced to showing that
\begin{equation*}
\bigg\|\int_0^t e^{i(t-s)\Delta}(V(\cdot)u(\cdot, t)) ds\bigg\|_{L_t^q([0,T]; L_x^r)}\lesssim
T^{\frac{a(q-2)}{q(a-2)}} \|V^{\frac{\alpha q(a-2)}{4(a-q)}}\|_{\mathcal{KS}_\alpha}^{\frac{4(a-q)}{\alpha q(a-2)}}\|u_0\|_{L^2}
\end{equation*}
under the same conditions on $\alpha$, $a$ and $(q,r)$ as in Theorem \ref{SS}.
By duality, it suffices to prove that
\begin{equation}\label{pop}
\bigg|\bigg\langle \int_0^t e^{i(t-s)\Delta}(Vu) ds , G \bigg\rangle_{x,t}\bigg| \lesssim |||V|||T^{\frac{a(q-2)}{q(a-2)}} \|u_0\|_{L^2} \|G\|_{L_t^{q\prime}([0,T]; L_x^{r'})}.
\end{equation}
To show this, we first write
\begin{align}\label{eq1}
\bigg\langle \int_0^t e^{i(t-s)\Delta}(Vu) ds , G \bigg\rangle_{x,t} &=  \int_{\mathbb{R}}\int_{0}^{t} \Big\langle Vu \,,\, e^{-i(t-s)\Delta}G \Big\rangle_{x}\, ds\,dt \nonumber\\
&= \bigg\langle  V^{1/2} u, V^{1/2} \int_{s}^{\infty} e^{-i(t-s)\Delta}G \,dt \bigg\rangle_{x, s},
\end{align}
and then use H\"older's inequality to bound \eqref{eq1} by
\begin{equation*}
\|u\|_{L_s^{q\prime}([0,T]; L_x^2(|V|))}  \bigg\|\int_{s}^{\infty} e^{-i(t-s)\Delta}G \,dt \bigg\|_{L_s^q([0,T]; L_x^2(|V|))} .
\end{equation*}
We will then show
\begin{equation}\label{sd2}
\bigg\|\int_{t}^{\infty} e^{i(t-s)\Delta}G\, ds \bigg\|_{L_t^q([0,T]; L_x^2(|V|))}
\lesssim  |||V|||^{1/2}T^{\frac{q-2}{q(a-2)}}  \|G\|_{L_t^{q\prime}([0,T]; L_x^{r'})}.
\end{equation}
Using this estimate together with \eqref{sd1}, the desired estimate \eqref{pop} follows.

To show \eqref{sd2}, we use the weighted estimate \eqref{shomo} and the dual estimate of \eqref{two-s} to see that
\begin{align*}
\bigg\|\int_{-\infty}^{\infty} e^{i(t-s)\Delta}G \,ds \bigg\|_{L_t^q([0,T]; L_x^2(|V|))} &= \bigg\|e^{it\Delta}\int_{-\infty}^{\infty} e^{-is\Delta}G \,ds \bigg\|_{L_t^q([0,T]; L_x^2(|V|))} \nonumber\\
&\lesssim |||V|||^{1/2}T^{\frac{q-2}{q(a-2)}} \bigg\|\int_{-\infty}^{\infty} e^{-is\Delta}G \,ds \bigg\|_{L_x^2}\nonumber\\
&\lesssim|||V|||^{1/2} T^{\frac{q-2}{q(a-2)}}\|G\|_{L_t^{q\prime}([0,T]; L_x^{r'})}
\end{align*}
for $2<q<a<\infty$, $n-1<\alpha<n$ and $n \ge3$.
Since $q>q'$ from $q>2$, we then use the Christ-Kiselev lemma (\cite{CK}) to conclude
\begin{equation*}\label{sck}
\Big\|\int_{-\infty}^{t} e^{i(t-s)\Delta}G ds \Big\|_{L_t^q([0,T]; L_x^2(|V|))}
\lesssim  |||V|||^{1/2}T^{\frac{q-2}{q(a-2)}}\|G\|_{L_t^{q\prime}([0,T]; L_x^{r'})},
\end{equation*}
which implies \eqref{sd2} by invoking the splitting of the integral
$\int_t^\infty=\int_{-\infty}^\infty-\int_{-\infty}^t$.

Once the local Strichartz estimates \eqref{ss} are shown for small $T$, they hold (with a constant depending on $T$)
for any $T$ by a straightforward iteration together with \eqref{mass}.

\subsubsection*{Acknowledgement}
The authors would like to thank the anonymous referees for their valuable comments
which helped to improve the manuscript.

\end{document}